\documentclass[12 pt]{article}
\usepackage[utf8]{inputenc}
\usepackage{mathtools}
\usepackage{amsmath, amscd, amssymb, amsthm,amsfonts}
\usepackage{euscript, mathrsfs, latexsym,mathabx,stmaryrd,hyphenat}
\usepackage{color,authblk}
\usepackage{multirow}
\usepackage{hyperref} 
\hypersetup{
  colorlinks   = true,    
  urlcolor     = blue,    
  linkcolor    = blue,    
  citecolor    = red      
}

\usepackage{tikz}
\usepackage{lmodern}
\usepackage[T1]{fontenc} 
\usepackage{multicol}
\usepackage{tikz,tikz-cd}\usetikzlibrary{matrix,arrows,calc}
\ifx\pdfoutput\undefined
\usepackage{graphicx}
\else
\fi

\usepackage{pst-node}
\usepackage{tikz-cd}

\ifx\xetex
\usepackage{fontspec}
\usefont{EU1}{lmr}{m}{n}
\else
\fi
\newtheorem{remark}{Remark}
\usepackage{url} 
\newtheorem*{theorem*}{Theorem}
\newtheorem{theorem}{Theorem}[section]

\newtheorem{corollary}{Corollary}[theorem]
\newtheorem{lemma}[theorem]{Lemma}

\numberwithin{subcase}{case}
\newcounter{Case}[theorem]
\refstepcounter{Case}

\usepackage{tikz}
\usetikzlibrary{matrix,arrows,positioning}
\tikzset{node distance=2cm, auto}

\usepackage[nodayofweek]{datetime}

\parskip=\medskipamount

\title{On the Uniqueness of Sarkar-Seed-Szabo Construction}
\author{Pravakar Paul}

\newcommand{\Addresses}{{
  \bigskip
  \footnotesize

  Pravakar Paul, \textsc{Department of Mathematics, The University Of Iowa,
    14 MacLean Hall, Iowa 52246}\par\nopagebreak
  \textit{E-mail address}, P.~Paul: \texttt{pravakar-paul@uiowa.edu}

}}

\date{}

\begin{document}

\maketitle
\begin{abstract}
\noindent In an attempt to consolidate Szabo's geometric chain complex and Bar\hyp Natan's chain complex, Sarkar-Seed-Szabo defined a total complex $CTot(L)$ over $\mathbb{F}_{2}$ by adding some extra terms $\{h_{i} \}_{i=2}^{\infty}$ in the differential. In this paper we prove the uniqueness of $\{h_{i} \}_{i=2}^{\infty}$ upto some properties. This in turn implies that the total complex $CTot(L)$ is unique.
\end{abstract}
\section{Introduction}
Khovanov categorified the Jones polynomial by constructing the invariant which is today called Khovanov Homology in \cite{MR1740682}. Since this time many variants of this foundational construction, called link homology theories have been studied, including Lee homology \cite{MR2173845}, Bar-Natan homology (\S 9.3 \cite{MR2174270}), Szabo's geometric perturbation \cite{MR3431667}, et cetera, in addition to many Floer theoretic constructions see \cite{MR2141852}, \cite{MR2764887}, \cite{MR2805599}. Each of these constructions can be viewed as a special case of a  general idea. Given a chain complex $(C,d)$, we can study its deformations by introducing a perturbed differential $d+\epsilon$. For $(C,d+\epsilon)$ to again be a chain complex, the Maurer\hyp Cartan equation must be satisfied: 
\begin{equation*}
[d:\epsilon]=\epsilon^2 
\end{equation*}

\noindent For a link homology theory $(C(L),d_{L}),$ it is natural for the deformation $\epsilon:= \epsilon_{L}$ to depend on the link $L$ in such a manner that the chain complex \[ (C(L), d_{L}+ \epsilon_{L} ) \] remains invariant under the Reidemeister moves. Such constructions give rise to spectral sequences (see \cite{MR2729272}, \cite{MR2822178}, \cite{MR3509974}, \cite{MR4092306}). Since there are many deformations of Khovanov homology, one can ask whether there are relations among them. Sarkar, Seed and Szabo studied the Bar-Natan homology and Szabo's geometric perturbation, producing a combination theory in \cite{MR3692911}. The existence of such a construction constitutes an example of one such relation. \\ \\
Both  Bar-Natan homology and Szabo's geometric perturbation can be retrieved from this combination theory by specialization. In more detail, if $(CKh(L), d_{1})$ denotes the Khovanov complex of a link and $H$ is a formal variable in $(gr_{h},gr_{q})$-bigrading $(0,-2)$ then,  Bar-Natan's chain complex is defined as:
\begin{equation}
    CBN(L) :=  \left (CKh(L) \otimes_{\mathbb{Z}} \mathbb{F}_{2}[H], \delta_{BN}=d_1+Hh_1 \right )
\end{equation}
\noindent where $\delta_{BN}$ is a differential with bigrading $(1,0)$.  
Szabo's geometric perturbation can also be stated with similar language. If $W$ is a formal variable in $(gr_{h},gr_{q})$ bigrading $(-1,-2)$ then, Szabo's geometric chain complex is defined as:
\begin{equation}
    CSz(L):= \left ( CKh(L) \otimes_{\mathbb{Z}} \mathbb{F}_{2}[W], \delta_{Sz}=\sum_{i=1}^{\infty} d_{i} W^{i-1} \right )
\end{equation}
\noindent where $\delta_{Sz}$ is a differential with bigrading $(1,0)$. 
In \cite{MR3692911} Sarkar, Seed and Szabo extended the Bar-Natan differential $h_1$ to construct an $(i,2i)$-graded endomorphisms $h_i$ for all $i \geq 2 $. They defined a new chain complex: 
\begin{equation}
 CTot(L):=  \left (CKh(L) \otimes_{\mathbb{Z}} \mathbb{F}_{2}[H,W],\delta_{Tot}= \sum_{i=1}^{\infty} d_iW^{i-1} + \sum_{i=1}^{\infty} HW^{i-1}h_{i} \right )    
\end{equation}

\noindent where the differential $\delta_{Tot}$ again has bigrading $(1,0)$. Thus, by setting $W=0$ or $H=0$, we can retrieve the Bar-Natan complex $ \left(CBN(L), \delta_{BN} \right )$ or Szabo's complex $ \left (CSz(L), \delta_{Sz} \right )$ respectively. We have the following diagram: 
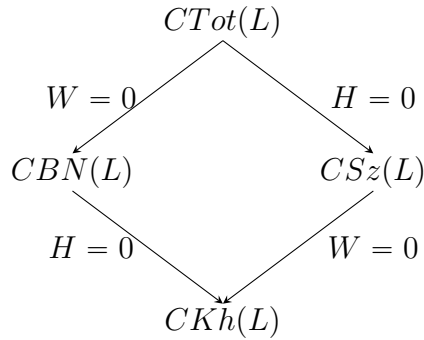
\begin{center}
\begin{tikzpicture}
\node at (0,0) {$CTot(L)$}; 
\node at (-2,-2) {$CBN(L)$};
\node at (2,-2) {$CSz(L)$};
\node at (0,-4) {$CKh(L)$};
\draw[-stealth] (0,-0.25)--(-2,-1.75);
\draw[-stealth] (0,-0.25)--(2,-1.75);
\draw[-stealth] (-2,-2.25)--(0,-3.75);
\draw[-stealth] (2,-2.25)--(0,-3.75);
\node at (-1.75,-3) {$H=0$}; 
\node at (2,-3) {$W=0$}; 
\node at (2,-1) {$H=0$}; 
\node at (-1.75,-3+2) {$W=0$}; 
\end{tikzpicture}
\end{center}

\noindent The existence of a combination theory begs the question of uniqueness of such construction. Does there exist $h_{i}'$ different than $h_i$ producing a separate combination theory $CTot^{'}(L)$? 
In this paper, we prove the uniqueness of the maps $h_{i}$, for $i \geq 2$. We show the map $h_i$ is unique under certain hypotheses which were used by Sarkar, Seed and Szabo in their original construction. It also raises the question of whether these hypotheses are sufficient for the uniqueness of such combination theories. This result is a corollary of our theorem below.  \\ 

\begin{theorem*}
For any  $k$-dimensional resolution configuration $C$, and any associated map $F_{C}: V_0(C) \to V_1(C)$ of bidegree $(k,2k)$ satisfying the naturality, the disoriented, the duality, the extension and the filtration rules  must agree with the Sarkar-Seed-Szabo formula. 
\end{theorem*}

 
 
\section*{Acknowledgements} 
I would like to thank my advisor Benjamin Cooper for suggesting the problem.

\section{Notations and terminology}
In this section we  define the notion of resolution configuration. We also explain how resolution configurations and the associated maps induce perturbations on the Khovanov complex. We list the rules or properties satisfied by Bar\hyp Natan's or Szabo's perturbation. These rules will be assumed in the main theorem. We will show that the rules determine the extended terms $h_{i}$ uniquely. 

\noindent A $k$\hyp dimensional \textbf{resolution configuration} is a set $C$ of disjoint circles $x_1,..,x_t$ in $S^{2}$ together with $k$ embedded arcs $\gamma_1,...,\gamma_{k}$, with the properties that 
\begin{enumerate}
    \item The arcs are disjoint from each other. 
    \item The endpoints of the arcs lie on the circles. 
    \item The inside of the arcs are disjoint from the circles. 
\end{enumerate} 
A resolution configuration is said to be \textbf{oriented} if the embedded arcs are oriented.
Given a resolution configuration $C=(x_1,..,x_t,\gamma_1,...,\gamma_{k})$ a $\textbf{decoration}$ is a choice of an orientation of all the arcs.  \\ 

\begin{tikzpicture}
\draw[blue] (-1,-1)--(1,1);
\draw[blue] (1,-1)--(0.1,-0.1);
\draw[blue] (-0.1,0.1)--(-1,1);
\draw[blue] (2+3,1) to [out=-30, in=30] (2+3,-1);
\draw[blue] (4+3,1) to [out=210, in=-210] (4+3,-1);
\draw[blue, ->] (5.5,0)--(6.5,0);
\node[above] at (6,0) {$\gamma$};
\node[above] at (6,-1.5) {$0$};
\node[above] at (11,-1.5) {$1$};
\draw[blue] (10,1) to [out=-30, in=-150] (12,1);
\draw[blue] (10,-1) to [out=30, in=150] (12,-1);
\draw[blue, ->] (11,-0.5-0.2)--(11,0.70);
\node[above] at (11.3,0) {$\gamma^{*}$};

\node[above] at (6,-3) {Figure $1$.Dual resolution configuration};
\end{tikzpicture}

\noindent Given an $k$\hyp dimensional oriented resolution configuration the \textbf{dual configuration} $C^{*}= (y_1,...,y_s, \gamma_1^{*},..., \gamma_{k}^{*})$ is given by the rule that transforms the $0$\hyp resolution into $1$\hyp resolution. More precisely, the dual circles $y_i$ are constructed from the $x$\hyp circles by making surgeries along the $\gamma$\hyp arcs, and the dual arcs $\gamma_{i}^{*}$ are given by rotating $\gamma_{i}$ by $90$ degrees counterclockwise. The circles 
$\{x_i \}$ are called the \textbf{starting circles} and the circles $\{y_i\}$ are called the \textbf{ending circles}. We fix the following notations:
\begin{align*}
 V_0(C) &= \bigotimes_{i=1}^{t} \mathbb{F}_2[x_i]/(x_i^2) &  V_1(C) &= \bigotimes_{j=1}^{s} \mathbb{F}_2[y_j]/(y_j^2)
\end{align*}
where the $gr_{q}$\hyp degree of $x$ and $1$ in $\mathbb{F}_{2}[x]/(x^2)$ are given by $-1$ and $1$ respectively. A $k$\hyp dimensional resolution configuration, $C$ will induce a map $F_{C}: V_{0}(C) \to V_{1}(C)$ of homological degree $k$ in the Khovanov complex i.e. for any monomials  $x\in V_0{(C)}$ and $y\in V_1(C)$, $gr_{h}(y)-gr_{h}(x)=k$.  
\noindent We will refer the elements of $V_{0}(C)$ or $V_{1}(C)$ as monomials. Given a configuration $C$, the $x_i$\hyp circles that are disjoint from all the $\gamma$\hyp arcs are called \textbf{passive circles}. The set of passive circles  for $C$ and $C^{*}$ agree.
A configuration is called \textbf{active}, when it has no passive circles. For any resolution configuration $C$, removing all the passive circles produces an active resolution configuration $C_0$. We have the following decomposition: 
\begin{align*}
V_0(C)&=V_0(C_0) \otimes P(C) &      V_1(C)=V_1(C_0) \otimes P(C)
\end{align*}
where $P(C)$ is the tensor product of ${\mathbb{F}}_2[w]/(w^2)$ of all passive circles $w$ of $C$.  \\ \\ 
Writing $S^{2}$ as $\mathbb{R}^2 \cup \{ \infty \}$, the \textbf{mirror image} of a resolution configuration $C$ is obtained by reflecting it along the line $\mathbb{R} \times \{ 0 \}$. 
The resolution configurations described above are used to construct the perturbed differentials $(2), (3)$ and $(4)$ from the introduction on the Khovanov complex by summing over maps of the form: \[ F_{C}: V_0(C) \to V_1(C) \] 
In particular, the map $d_i$ is a sum over all $F_{C}$ for all $i $\hyp dimensional oriented resolution configurations $C$. 
\begin{align*}
d_{i}=\sum_{|C|=i} F_{C}    
\end{align*}
Often more sophisticated notation is used to express this sum, see p.10 \cite{MR3431667}. To be precise, in order to construct differentials such as $\delta_{Sz}$ or $\delta_{BN}$ combinatorially it suffices to define maps $F_{C}$ associated to resolution configurations. \\ \\
\noindent For a given collection of $\{ F_{C} \}$ to be reasonable, the following rules are used as guides.
Next, we list the set of rules relevant to Bar-Natan and Szabo's perturbation. \\
\begin{enumerate}

    \item \noindent \textsl{\large{Naturality Rule:}} Let $C$ and $C'$ be two $k$\hyp dimensional configurations, with the property that there is an orientation preserving diffeomorphism of the sphere that sends $C$ to $C'$ . Then the diffeomorphism induces canonical identifications $V_0(C)=V_0(C')$, and $V_1(C)= V_1(C')$. Under these identifications \[ F_C= F_{C'} \]

    \item \noindent \textsl{\large{Conjugation Rule:}} If $r(C)$ denotes the resolution configuration where all the orientations of the edges in $C$ are reversed then, \[ F_{C}= F_{r(C)} \]
   Given a circle $z$, we define the duality endomorphism on the $2$\hyp dimensional vector space $V(z)=\mathbb{F}_{2}/(z^2)$ by $1^{*}=z$ and $z^{*}=1$. This induces the duality map on $V_{0}(C)$ and $V_{1}(C)$. There is a canonical identification between $V_{0}(C)$ with $V_{1}(m(C^{*}))$ and $V_{1}(C)$ with $V_{0}(m(C^{*}))$ where $m(C^{*})$ denotes the mirror of the dual configuration.
    \item \noindent \textsl{\large{Duality Rule:}}
    Let $m(C^{*})$ be the mirror of the dual configuration. Then for all pairs of monomials $(a,b)$, where $a \in V_0(C)$ and $b \in V_1(C)$, the coefficient of $F_{C}(a)$ at $b$ is equal to the coefficient of $F_{m(C^{*})}(b^{*})$ at $a^{*}$. More precisely, any isomorphism $\varphi$ between $\mathbb{F}_{2}[x]/(x^2)$ and $\left( \mathbb{F}_{2}[x]/(x^2) \right )^{*}$ extends to an isomorphism $\tilde{\varphi}$ between $V_{i}(C)$ and $V_{i}(C)^{*}$ for $i=0,1$. If $\varphi$ is chosen to be 
    \begin{gather*}
\varphi:\mathbb{F}_{2}[x]/(x^2) \to \left(\mathbb{F}_{2}[x]/(x^2) \right)^{*} \notag \\
\varphi(x)=1^{*}  \,\, \text{and}\quad \varphi(1)=x^{*}
    \end{gather*}
   then the duality rule ensures the commutativity of the following diagram:
\catcode`\@=11
\newdimen\cdsep
\cdsep=3em

\def\cdstrut{\vrule height .6\cdsep width 0pt depth .4\cdsep}
\def\@cdstrut{{\advance\cdsep by 2em\cdstrut}}

\def\arrow#1#2{
  \ifx d#1
    \llap{$\scriptstyle#2$}\left\downarrow\cdstrut\right.\@cdstrut\fi
  \ifx u#1
    \llap{$\scriptstyle#2$}\left\uparrow\cdstrut\right.\@cdstrut\fi
  \ifx r#1
    \mathop{\hbox to \cdsep{\rightarrowfill}}\limits^{#2}\fi
  \ifx l#1
    \mathop{\hbox to \cdsep{\leftarrowfill}}\limits^{#2}\fi
}
\catcode`\@=12

\cdsep=3em

$$\begin{matrix}
  V_{0}(m(C^{*}))                    & \arrow{r}{F_{m(C^*)}}   & V_{1}(m(C^{*}))                    \cr
  \arrow{d}{\tilde{\varphi}} &                      & \arrow{d}{\tilde{\varphi}} \cr
  V_{1}(C)^{*}                  & \arrow{r}{F_{C}^{*}} & V_{0}(C)^{*}                  \cr

\end{matrix}$$
\label{dual}
    
    \item\noindent \textsl{\large{Filtration Rule:}}  Let $C$ be a configuration and $a \in V_0(C)$, $b \in V_1(C)$ be two monomials. For a point $P$ in the union of the starting circles, let $x(P)$ and $y(P)$ denote the starting and the ending circles that contain $P$. If $a$ is divisible by $x(P)$ and the coefficient of $F_{C}(a)$ at $b$ is non-zero, then $b$ is divisible by $y(P)$. 
    
    \item \noindent \textsl{\large{Disoriented Rule:}} If $C$ and $C'$ differ in the orientation of arcs then, \[ F_C=F_{C'} \] 
    
    \item \noindent \textsl{\large{Extension Rule:}} For a configuration $C$, the associated map $F_{C}$ depends only on the active part $C_0$ i.e. \[ F_C(a\cdot v)=F_{C_0}(a) \cdot v \]
where $v \in P(C)$ and $a \in V_0(C)$. \\

    \item \textsl{\large{Disconnected Rule:}} If $C$ is a disconnected configuration i.e. the active circles of $C_0$ can be partioned into two non\hyp empty sets, $c_1,...,c_s$ , $d_1,...,d_t$ so that none of the $\gamma$ arcs connect $c_i$ to $d_j$ then \[ F_{C} =0 \]

\end{enumerate}

\noindent Sometimes these rules are imposed depending upon how the construction deals with the orientation. The naturality rule asserts invariance under isotopy of $S^{2}$, while conjugation and duality rules are TQFT duality properties. The extension rule and the filtration rule correspond to partial functoriality. The filtration rule allows us to define a reduced theory. Below we make a small table for comparisons between different theories:

\begin{center}
\begin{tabular}{ |c||c| } 
 \hline
 Khovanov's $d_1$ & 1,2,3,4,5,6 see, Def 2.4 \cite{MR3692911} \\  
 \hline
 Bar-Natan's $h_1$ & 1,2,3,4,5,6  see, Def $2.5$, \cite{MR3692911} \\ 
 \hline
 Szabo's $d_i$ for $i \geq 2$ & 1,2,3,4,6,7 see, Def 2.6 \cite{MR3692911} \\ 
 \hline
\end{tabular}
\end{center}

\noindent This section concludes with the definition of the Sarkar-Seed-Szabo's extended differential $h_i,$ for $i \geq 2 $.
Before we define the extended terms $h_i$, we need to define the words \textbf{tree} and \textbf{dual tree}. \\ 

\noindent A connected resolution configuration of dimension $k$ is said to a \textbf{tree} if it has exactly $k+1$ starting circles and $1$ ending circle. A dual of a tree is called a \textbf{dual tree}. 
$F_{C}$ is \textbf{non-zero} if and only if $C$ is a disjoint union of \textbf{trees} and \textbf{dual trees}.
If $C$ is a connected $k$\hyp dimensional resolution configuration which is a tree $\{ x_1,..,x_{k+1}, \gamma_1,...,\gamma_{k} \}$
then, define
\begin{equation}
\label{tree}
F_{C}(\prod_{i=1}^{k+1}x_i)=y     
\end{equation}

\noindent where $y$ is the ending circle. On any other monomial $F_{C}$ is defined to be zero. Dually, if C is a dual tree $\{y,\gamma_1^{*},...,\gamma_{k}^{*} \}$ then the only non-zero term of $F_{C}$ is given by
\begin{equation}
\label{dual-tree}
 F_{C}(1)=1\otimes ... \otimes 1 .    
\end{equation}

\noindent More generally, if $C$ is a disjoint union of \textbf{trees} and  \textbf{dual trees}: $C= \{C_{i} \}_{i=1}^{k} \bigsqcup \{ {C^{'}}_{j}^{*} \}_{j=1}^{l} $ then define:

\begin{equation}
F_{C}=( \bigotimes_{i=1}^{k} F_{C_{i}}) \bigotimes(\bigotimes_{j=1}^{l} F_{{C^{'}}_{j}^{*}}).
\end{equation}

. \\ \\
\begin{remark}
Recall that in the Khovanov complex the quantum grading is normalized to make $d_1$ into a differential in bidegree $(1,0)$. For a $k$\hyp dimensional resolution configuration, the normalized quantum grading convention increases the quantum grading of monomials in $V_{1}(C)$ by $k$ relative to the quantum grading of monomials in $V_{0}(C)$. Thus, 
Sarkar-Seed-Szabo's extended term $\{ h_k, k \geq 2 \}$ has bigrading $(k,2k)$ and it satisfies the naturality rule, the disoriented rule, the duality rule, the extension rule and the filtration rule (Lemma 3.3, \cite{MR3692911}).
\end{remark}

\section{Proof of the Theorem}
We prove the main theorem by the well\hyp ordering principle. In Lemma 3.1, we prove the theorem for all $2$\hyp dimensional connected resolution configurations. Next, we will choose the $k \geq 3 $ minimal with respect to the property that there exists a family of connected $k$\hyp dimensional resolution configurations other than trees and dual trees that defines a map in bidegree $(k,2k)$ satisfying the naturality rule, the filtration rule and the duality rule. This will allow us to construct a family of $k-1$\hyp dimensional connected resolution configurations other than trees and dual trees contradicting the minimality hypothesis on $k$. The disconnected case will then follow from the connected case.        
\begin{lemma}
If $C$ is a connected $2$\hyp dimensional unoriented configuration, then any non\hyp zero map \[ F_{C}: V_0 (C) \rightarrow V_1 (C) \] satisfying the naturality rule, the disoriented rule, the filtration rule and the duality rule in bidegree $(2,4)$ agrees with the assignment in $(5)$ and $(6)$. In more detail, it must be equal to $zero$ unless $C$ is a tree or a dual tree. In this case, $F_{C}(x_1x_2x_3)=y $ when $C$ is a tree and $F_{C}(1)=1\otimes 1 \otimes 1$ when $C$ is a dual tree where, $x_1,x_2,x_3$ are the starting circles and $y$ is the ending circle of the tree.

\begin{proof}

Szabo listed all the connected $2$\hyp dimensional configurations in (see p.$8$, \cite{MR3431667}). They are enumerated as follows:  \\

\begin{tikzpicture}

\draw [blue] (2,2) circle [radius=0.5];;
\draw[blue] (2.5,2.1875)--(3.5,2.1875);
\draw[blue] (2.5, 1.8125)--(3.5,1.8125);
\draw [blue] (4,2) circle [radius=0.5];;
\node[below] at (3,1) {1} ;

\draw[blue] (6,2.5) circle [radius=0.25];;
\draw[blue] (8,2.5) circle [radius=0.25];;
\draw[blue] (7,1.5) circle [radius=0.25];;
\draw[blue] (6,2.25)--( 6.75,1.5);
\draw[blue] (8,2.25)--(7.25,1.5);
\node[below] at (7,1) {2};

\draw[blue] (10,2.5) circle [radius=1] ;;
\draw[blue] (10,2.5) circle [radius=0.25];;
\draw[blue] (12,3.5) circle [radius=0.25];;
\draw[blue] (10.25,2.5)--(11,2.5);
\draw[blue] (11.75,3.5)--(10.75,3.161);
\node[below] at (10,1) {3} ;

\draw[blue] (14,2.5) circle [radius=1];; 
\draw[blue] (13.25,3.161)--(14.75,3.161);
\draw[blue] (13.25,1.8385)--(14.75,1.8385);
\node[below] at (14,1) {4}; 
\end{tikzpicture}

\begin{tikzpicture}

\draw[blue] (2,2) circle [radius=1] ;;
\draw[blue] (1.25,2.6614)--(2.75,2.6614);
\draw[blue] (1.25,1.3385) to [out=-135 , in=-180]  (2,0.5);
\draw[blue] (2,0.5) to [out=0, in=-45] (2.75,1.3385);
\node[below] at (2,0) {5};

\draw[blue] (6,1) circle [radius=1];; 
\draw[blue] (6,3) circle [radius=0.25];;
\draw[blue] (5,1)--(7,1); 
\draw[blue] (6,2)--(6,2.75);
\node[below] at (6,0) {6};

\draw[blue] (10,2) circle [radius=1];;
\draw[blue](10,2) circle [radius=0.25];;
\draw[blue] (10,3)--(10,2.25);
\draw[blue] (9.25,1.3385)--(10.75,1.3385);
\node[below] at (10,0) {7};

\draw[blue](14,2) circle [radius=1];; 
\draw[blue] (13,2)--(15,2);
\draw[blue]  (14,3) to [out=60, in=90] (15.5,2);
\draw[blue] (15.5,2) to [out=-90, in=-60] (14,1);
\node[below] at (14,0) {8} ; 

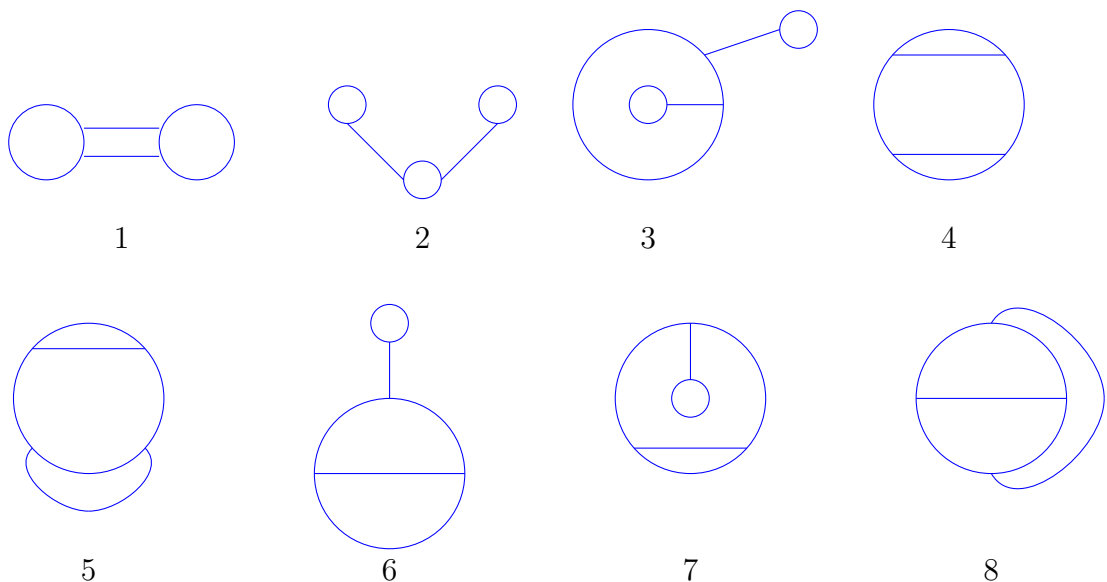
\node[below] at (8,-1) {Figure 2: Types of connected $2$\hyp dimensional configurations};

\end{tikzpicture}

\noindent We observe that configurations $ 1, 6, 7, 8$ are neither trees nor dual trees. We will show that $F_{C}$ must be equal to zero in these cases. Configurations $2, 3, 4, 5$ are tress or dual trees. Configurations $(1,2), (3,4)$, and $(5,6)$ are dual to each other while the configurations $7$ and $8$ are self duals.

\noindent We study each of these configurations separately.  \\
\begin{description}
\item[\underline{Configuration $1, 6$ and $7:$}]
Here, we have two starting circles $x_1,x_2$ and two ending circles $y_1,y_2$. Thus, \[ F_{C}: V(x_1) \otimes V(x_2)\rightarrow V(y_1) \otimes V(y_2) \] As $F_{C}$ has bidegree $(2,4)$, we have the following constraints: \\
\begin{eqnarray*}
F_{C}(x_1 \cdot x_2) & \in & \langle y_1, y_2 \rangle \\ 
F_{C}(x_i) & \in &  \langle 1 \rangle \\
F_{C}(1) & = & 0
\end{eqnarray*}
If $F_{C}$ satisfies the filtration rule, then $y_{1} \cdot y_2| F_{C}(x_1 \cdot x_2) $ and $y_j|F_{C}(x_i)$ where $x_i$ and $y_j$ contain a common point P. This implies $F_{C} \equiv 0 $ in this cases. 

\item[\underline{Configuration $8:$}]
Here, we have one starting circle $x$ and one ending circle $y$. Thus, \[ F_{C}: V(x) \rightarrow V(y) \] As $F_{C}$ has bidegree $(2,4)$, we have the following constraints on $F_{C}:$

\begin{eqnarray*}
F_{C}(x) & \in & \langle 1 \rangle \\
F_{C}(1) &= & 0
\end{eqnarray*}

\noindent But this is impossible as $F_{C}$ must satisfy the filtration rule, as before, unless $F_{C}$ is zero. 

\item[\underline{Configurations $2$ and $3$:}]
These configurations are trees. There are three starting circles $x_1,x_2,x_3$ and one ending circle $y$. Thus, 

\[ F_{C}: V(x_1) \otimes V(x_2) \otimes V(x_3)\rightarrow V(y)  \]
The bidegree $(2,4)$ constraint on $F_{C}$ imply the following:

\begin{eqnarray*}
    F_{C}(x_1 \cdot x_2 \cdot x_3) & \in & \langle y \rangle \\
    F_{C}(x_i \cdot x_j) & \in &  \langle 1 \rangle \\
    F_{C}(x_i) & = & 0 , \\
    F_{C}(1) & = & 0  
\end{eqnarray*}
We rule out the second constraint as $F_{C}$ satisfies the filtration rule. Thus, $F_{C}(x_1\cdot x_2\cdot x_3)=y$ is the only non-zero term.
\item[\underline{Configurations $4$ and $5$:}]

We are left with configurations $4$ and $5$.
We can either use duality rule or argue as follows:
configurations $4$ and $5$ are dual trees. We have one starting circle $x$ and three ending circles $y_1,y_2,y_3$. Thus, 

\[ F_{C}: V(x) \rightarrow V(y_1) \otimes V(y_2) \otimes V(y_3)  \] 
The bidegree constraint on $F_{C}$ implies the following: 
\begin{eqnarray*}
F_{C}(1) & \in & \langle 1 \rangle \\
F_{C}(x) & = & 0
\end{eqnarray*}

\noindent Hence, $F_{C}(1)=1$ is the only non-zero term in this case. 

\end{description}

\end{proof}
\end{lemma}

\noindent In the next lemma, we will show that any associated map $F_{C}$ for a resolution configuration $C$ satisfying the filtration rule and the naturality rule can be extended to an associated map satisfying the filtration rule, the duality rule and the naturality rule.  \\ \\ 

\noindent Let $C$ be a self\hyp dual unoriented resolution configuration i.e. $C=m(C^{*})$. Let $\psi$ denotes a diffeomorphism of $S^{2}$ that identifies $C$ with $m(C^{*})$.  Suppose, $F_{C}: V_{0}(C) \to V_1(C)$ is an associated map satisfying the naturality rule. The commutative diagram in the duality rule can be used to define a map $F_{m(C^{*})}$ on the mirror of the dual resolution configuration $m(C^{*})$. As $\psi$ identifies $V_{0}(C)$ with $V_{0}(m(C^{*}))$ and $V_{1}(C)$ with $V_{1}(m(C^{*}))$, $F_{m(C^{*})}$ defines a map $F'_{C} : V_{0}(C) \to V_1(C)$ via the identification $\psi$ i.e. we have the following commutative diagram: 
\begin{center}
\begin{tikzpicture}
 \node at (0,0) {$V_{0}(C)$}; 
 \draw[-stealth] (0.75,0)--(3,0); 
 \node at (3.75,0) {$V_{1}(C)$};
 \draw[-stealth] (0,-0.5) -- (0,-2.75);
 \node at (0.25,-1.75) {$\psi$};
  \node at (0.25+3.75,-1.75) {$\psi$};
   \node at (0.25+3.75,-1.75-3) {$\tilde{\varphi}$};
   \node at (0.25,-1.75-3) {$\tilde{\varphi}$};
  \draw[-stealth] (3.75,-0.5) -- (3.75,-2.75);
  \node at (-1.25,-3) {$V_1(C)=V_{0}(m(C^{*}))$};
   \draw[-stealth] (0.75,-3)--(3,-3); 
    \node at (5,-3) {$V_{1}(m(C^{*}))=V_{0}(C)$}; 
 \draw[-stealth] (0,-0.5-3) -- (0,-2.75-3); 
   \draw[-stealth] (3.75,-0.5-3) -- (3.75,-2.75-3);  
   \node at (0,-6) {$V_{1}(C)^{*}$}; 
  \node at (3.75,-6) {$V_{0}(C)^{*}$}; 
  \draw[-stealth] (0.75,-6)--(3,-6);
  \node at (1.875,-5.75) {$F_{C}^{*}$};
   \node at (1.875,-2.75) {$F_{m(C^{*})}$};
   \node at (1.875,-2.75+3) {$F'_{C}$};   
\end{tikzpicture}
\end{center}

We define, 
\begin{equation*}
    \tilde{F}_{C}= 
    \begin{cases*}
     F_{C}+F'_{C} & if $F_{C} \neq F'_{C}$ \\
     F_{C}        & otherwise 
    \end{cases*}
\end{equation*}

\begin{lemma}
Let $C$ be an unoriented resolution configuration and let $\{x_i\}$ and $\{y_j \}$ denote the starting and the ending circles of $C$ respectively. Let, \[F_{C}: V_{0}(C) \to \bigotimes V_{1}(C) \] is an associated map which satisfies the filtration rule and the naturality rule. Define: \[ F_{m(C^*)}: V_{0}(m(C^{*})) \to \bigotimes V_{1}(m(C^{*})) \] on the mirror image of dual configuration $m(C^{*})$ as given by the duality rule. If $C$ is self\hyp dual i.e. $C=m(C^{*})$ then, replace $F_{C}$ by $\tilde{F}_{C}$. $F_{C}$ when $C \neq m(C^{*})$ or $\tilde{F}_{C}$ when $C=m(C^{*})$  satisfy the filtration rule, the duality rule and the naturality rule. 

\begin{proof}
$\tilde{F}_{C}$ or $F_{C}$ satisfy the duality rule by construction. To complete the proof we need to show that $\tilde{F}_{C}$ or $F_{m(C^{*})}$ satisfy the filtration rule. 
For a point $P$, let $x(P)$ and $y(P)$ denote the starting circle and ending circle of $C$ containing the point $P$ respectively. Let $b \in \bigotimes V(y_j)$ and $a \in \bigotimes V(x_i)$ are two monomials such that the coefficient of $F_{m(C^{*})}(b)$ at $a$ is nonzero. This implies that the coefficient of $F_{C}(a^{*})$ at $b^{*}$ must be nonzero. If we assume $y(P) \mid b $, then we want to show that $x(P)\mid a$. Equivalently, it would be sufficient to show that $x(P) \nmid a^{*} $.  \\ 
Thus, for the sake of contradiction, we may assume $x(P) \mid a^{*}$. By the filtration rule on $F_{C}$, this implies $y(P)$ divides all the monomials appearing on $F_{C}(a^{*})$. In particular $y(P) \mid b^{*}$, which is a contradiction. 

For $\tilde{F_{C}}$ we notice that $F'_{C}=\psi^{-1} \circ  F_{m(C^{*})}    \circ \psi$ also satisfies the filtration rule as $\psi$ is defined using a diffeomorphism of $S^{2}$. Hence, $\tilde{F}_{C}$ satisfies the filtration rule because it is equal to the sum of two maps satisfying the filtration rule.  
\end{proof}

\end{lemma}

\begin{lemma}
If $C$ is a $n$\hyp dimensional tree or dual tree and $F_{C}$ is any non\hyp zero associated map of bidegree $(n,p)$ satisfying the naturality rule, the filtration rule and the duality rule.\\
$\bullet$If $p=2n$ then, it must agree with Sarkar-Seed-Szabo formula in (\ref{tree}) and (\ref{dual-tree}).\\
$\bullet$If $p>2n$ then, $F_{C}$ must be zero.  
\end{lemma}
\begin{proof}
Because of the duality rule, it would be sufficient to focus on tree. Let $\{ x_1,x_2,...,x_{n+1} \}$ and $\{ y \}$ denote the set of starting circles and ending circle respectively. When $p=2n$, the degree restriction on $F_{C}$ implies that the possible non\hyp zero terms must satisfy: 
\begin{eqnarray*}
F_{C}(\prod_{i=1}^{n+1} x_{i}) & \in & \langle y \rangle \\ 
F_{C} (x_1\cdot \cdot \cdot \hat{x_i}\cdot \cdot \cdot x_{n+1}) & \in & \langle 1 \rangle 
\end{eqnarray*}
\noindent If $F_{C}$ satisfies the filtration rule, then for any monomial $m$ with $x_i|m$ will imply $y|F_{C}(m)$. In particular, $F_{C} (x_1\cdot \cdot \cdot \hat{x_i}\cdot \cdot \cdot x_{n+1})$ must be divisible by $y$. Hence, $(\ref{tree})$ and $(\ref{dual-tree})$ are the only non\hyp zero terms. \\ 

\noindent In the other case, when $p>2n$ we compute the difference between the maximum quantum degree in $V(y)$ and the minimum quantum degree in $\bigotimes_{i=1}^{n+1} V(x_i)$ which is,  
\begin{align*}
    gr_{q}(1)-gr_{q}(\prod_{i=1}^{n+1} x_{i}) &= gr_{q}(1)+n-(n+1)gr_{q}(x); \, \, \textit{see Remark 1.}\\
     &= 1+n-(-n-1) \\ 
     &=2n+2
\end{align*}
Thus, the only non\hyp zero map $F_{C}$ of bidegree $(n,p)$ with $p>2n$ must be equal to $F_{C}(\prod_{i=1}^{n+1} x_{i})=1$ which is a contradiction to the filtration rule. 
\end{proof}

\begin{theorem}
If $C$ is a connected $n$\hyp dimensional unoriented configuration which is not a tree or a dual tree, then any $F_{C}: V_{0}(C) \rightarrow V_1(C)$ satisfying the filtration rule, the duality rule and the naturality rule of bidegree $(n,p)$ with $p \geq 2n $ must be equal to $zero$. 
\end{theorem}

\begin{proof}
We divide the proof into several cases. When $n=2$ this case follows from lemma $3.1$. Next, we choose $n \geq 3$ that is minimum and contradicts the theorem. \\ 
\begin{description}

\item{\textbf{Case 1.}}
Suppose $C$ contains a single starting circle $x$. As $C$ is not a dual tree, the number of ending circles must be strictly less than $n+1$ . Figure $3$ illustrates this when $n=4$. Let $s$ denotes the number of ending circles with $s < n+1 $. We have the associated map
\[ F_{C}: V(x) \rightarrow \bigotimes_{j=1}^{s} V(y_j) \] 
If $F_{C}(x) \neq 0 $, then $F_{C}(x)= \prod_{j=1}^{s} y_j$ as $F_{C}$ satisfies the filtration rule. But then

\begin{eqnarray*}
gr_{q}(\prod_{j=1}^{s} y_j ) - gr_{q}(x) &=& \left( s \cdot gr_{q}(x) +n \right) - \left( 1 \cdot gr_{q}(x) \right); \, \; \textit{see Remark 1.} \\
&=& \left( -s+n \right) -(-1) \\
 &=& -s+n+1
\end{eqnarray*}
 \noindent We get a contradiction as $ \left (-s+n+1 \right)$ is strictly less than $p \geq 2n $ which is the quantum degree of $F_{C}$. Thus, we may assume $F_{C}(x)=0$ and $F_{C}(1) \neq 0 $. Let $ F_{C}(1)= \sum  y_{j_1} y_{j_2}... y_{j_m} ,0 \leq m \leq s $. We again look at the difference: 
\begin{eqnarray*}
gr_{q}(y_{j_1} y_{j_2}... y_{j_m}) - gr_{q}(1)  &=&
\left( m \cdot gr_q(y_{j_{k}})+(s-m)\cdot gr_{q}(1) +n  \right) - (1 \cdot gr_{q}(1)) \\
&=& \left (-m+(s-m)+n \right) - (+1) \\
&=& s+n-2m-1
\end{eqnarray*}
If $gr_{q}(F_{C})=p$, then we must have 
 \begin{eqnarray*}
 s+n-2m-1 &=&  p \\ 
 s &=& p+2m+1 
\end{eqnarray*} 
 It leads to a contradiction as we've assumed that $s<n+1$ above.  \\ \\  

\begin{tikzpicture}
\draw[blue] (2,2) circle [radius=1] ;;
\draw[blue] (2.9949, 2.25)--(1.005,2.25); 
\draw[blue](2.9949,1.75)--(1.005,1.75);
\draw[blue] (2.866,2.5) to [out=45, in=90] (4,2); 
\draw[blue] (4,2) to [out=-90, in=-45] (2.866,1.5);
\draw[blue] (1.13397,2.5) to [out=135, in=90] (0,2);
\draw[blue] (0,2) to [out=-90,in=-135] (1.13397,1.5);
\node[below] at (2,1) {Starting circle};

\draw[blue] (10,2) circle [radius=1.5];;
\draw[blue] (10,2) circle [radius=0.75];;
\draw[blue] (10,2) circle [radius=0.25];;
\node[below] at (10,0.25) {ending circles};

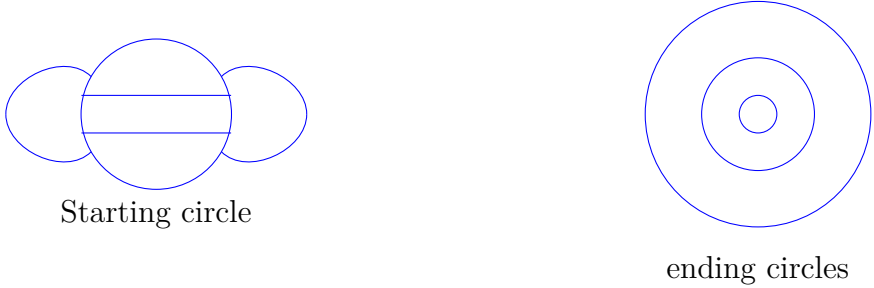
\node[below] at (6,-1) {Figure $3$: The number of ending circles is strictly less than $5$};
\end{tikzpicture}

\item{\textbf{Case 2.}} Suppose $C$ contains a single ending circle. By case $1$, we know 
the associated map for the dual configuration $F_{m(C^{*})} \equiv 0$. Since $F_{C}$ satisfies the duality rule, we have $F_{C} \equiv 0 $.

\item{\textbf{Case 3.}}
 Suppose $C$ contains more than one starting and ending circle. Let $t,s$ denote the number of starting and ending circles with $ 2 \leq t,s \leq n $. Any map $F_{C}$ must be of the form 
 \[ F_{C}: V_{0}(C)= \bigotimes_{i=1}^{t} V(x_i)\rightarrow V_1(C)= \bigotimes_{j=1}^{s} V(y_j) \] \\
 As before, we first argue for  $x= \prod_{i=1}^{t} x_i$. As $F_{C}$ satisfies the filtration rule, $F_{C}(x) \neq 0$ implies, $y_j| F_{C}(x), \forall j$. Thus,  \[F_{C}(x)= F_{C}(\prod_{i=1}^{t}x_{i})= \prod_{j=1}^{s} y_j  \] Then,
 \begin{eqnarray*}
 gr_{q}(\prod_{j=1}^{s} y_j) - gr_{q}(\prod_{i=1}^{t}x_{i}) & = & \left( s \cdot gr_{q}(y_{j})+n \right) -(t \cdot gr_{q}(x_i)); \textit{see Remark 1.} \\ 
 &=& \left( s \cdot(-1)+n  \right) - (t \cdot (-1))\\
  & = & t-s+n 
 \end{eqnarray*}
It leads to a contradiction as 
 \begin{eqnarray*}
t-s+n & = & p \\ 
t-s &= & p-n  
 \end{eqnarray*}

\noindent But $2 \leq t,s \leq n$ implies $|t-s|<n$ which is a contradiction as $p-n \leq n$.  Thus, we conclude $F_{C}(\prod_{i=1}^{t}x_{i})=0$.\\ \\
Next, we claim $h(1)=0$. By duality rule, we get rid of the possibility $F_{C}(1)=1$. So, assume $F_{C}(1)= \sum y_{j_1} y_{j_2}\cdot \cdot \cdot y_{j_m}, 1<m\leq s $. Then,
\begin{eqnarray*}
gr_{q}(y_{j_1} y_{j_2}\cdot \cdot \cdot y_{j_m}) -gr_{q}(1) &=& \left( m \cdot gr_{q}(y_{j_{k}})+(s-m)\cdot gr_{q}(1) +n \right)- \left(t \cdot gr_{q}(1) \right); \\
&& \textit{see Remark 1.} \\
&=& \left(m\cdot (-1)+(s-m) \cdot(+1)+n \right )- t \cdot (+1) \\
&=& s+n-2m-t
\end{eqnarray*}
If $gr_{q}(F_{C})=p$, then 
\begin{eqnarray*}
s+n-2m-t & =&  p \\
s & = & p-n+2m+t    
\end{eqnarray*}
But $2\leq t,s \leq n$ implies $s \leq n$ and since $1 < m \leq s$, the equation above implies $s>n$ which is a contradiction. 
\noindent  Hence, we have \[ F_{C}(1) =0 \, \,\textit{and} \, \, F_{C}(\prod_{i=1}^{t} x_i)=0.\]  
 Finally, we claim $F_{C}(x)=0$ for any monomial $x \neq \prod_{i=1}^{t}x_i ,1 $. As the configuration is assumed to be connected, there exists two circles in the starting circles $x_k,x_l$ that are connected by an arc where $x_k \mid x $ and $x_l \nmid x$.
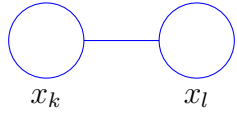
 \begin{center}
\begin{tikzpicture}

\draw [blue] (2,2) circle [radius=0.5];;
\draw[blue] (2.5,2)--(3.5,2);
\draw [blue] (4,2) circle [radius=0.5];;
\node[below] at (2,1.5) {$x_k$} ; 
\node[below] at (4, 1.5) {$x_l$} ; 
\end{tikzpicture}
\end{center}
\noindent Suppose $C'$ denote the connected $(n-1)$\hyp dimensional configuration that is constructed by performing surgery along the arc that connects $x_k$ and $x_{l}$. We notice that both $C$ and $C'$ have the same ending circles $ \{ y_j \}$. The monomial $(x_k)^{-1} x $  which consists of all the circles of $x$ except $x_k$ can naturally be thought of as an element of $V_0(C')$. Define the maps $F_0, F_1$ as follows: 
\begin{equation}
\begin{array}{cc}
      V_{0} (C) \xrightarrow{F_0}V_0(C') \xrightarrow{F_1}   V_1(C')= V_1(C)  \\ \\ 
      F_{0}(x) \coloneqq (x_{k})^{-1}x \\ \\
      F_1( (x_k)^{-1} x)\coloneqq F_{C}(x)
\end{array}
  \end{equation}

\noindent We define $F_0=0$ and $F_1 =0$ for other monomials in $V_{0}(C)$ and $V_{0}(C')$ respectively. First, we notice that \[gr_{q}((x_k)^{-1} x)- gr_{q}(x)= 2.\] Thus, the bidegree of $F_0$ is equal to $(1,2)$. As $F_{C}(x) = F_1\circ F_0 (x) $ by construction, we conclude that the bidegree of $F_1$ is equal to $(n-1, p-2)$. Let $x(P)| (x_{k})^{-1}\cdot x$ is the starting circle and $y(P)$ is the ending circle containing the point $P$ of $C'$. Now, $x(P)$ can naturally be identified with a starting circle $x_1(P)$ of $C$. Observe that $x_1(P) \neq x_k,x_l$ as none of the circles $x_k$ and $x_l$ appear in the expression $x_{k}^{-1}\cdot x$. Because $F_{C}$ satisfies the filtration rule and both $C,C'$ have the same ending circles, we conclude that $y(P)| F_{C}(x)=F_{C'}((x_{k})^{-1}\cdot x) $. Hence, $F_{C'}$ satisfies the filtration rule. $F_{1}$ can be extended to a map $\tilde{F_{1}}$ satisfying the filtration rule and the naturality rule. Finally, by lemma $3.2$  we get a bidegree $(n-1,p-2)$ map satisfying the filtration rule, the duality rule and the naturality rule. \\ \\ 
If $C'$ is a resolution configuration other than a dual tree, then we would have a counter\hyp example of dimension less than $n$, which contradicts the minimality of $n$. Thus, we may assume $C'$ is a dual tree as the number of ending circles of $C'$ is equal to $t \geq 2$ . As, $F_{1}((x_{k})^{-1}x) \neq 0 $ and $C'$ is a dual tree by lemma $3.3$, we conclude $(x_{k})^{-1}x$ must be equal to $1$. Thus, $x=x_{k}$. Lemma $3.3$ also implies,
\begin{eqnarray*}
F_{C}(x) & = & F_{C}(x_k) \\ 
&=& \tilde{F_1}(F_{0}(x_{k}))\\
&=& \tilde{F_{1}} (x_{k}^{-1}x_{k})\\
& = & \tilde{F_{1}} (1) \\ 
&= & 1
\end{eqnarray*}
It violates the hypothesis that $F_{C}$ satisfies the filtration rule which contradicts our assumption.  

\end{description}
\end{proof} 





\noindent Finally we can prove our main theorem in which $C$ is allowed to be disconnected. 
\begin{theorem}

Suppose $C$ is a $k$\hyp dimensional resolution configuration and $F_{C}: V_0(C) \to V_1(C)$ is the associated map in bidegree $(k,2k)$   satisfying the naturality, the disoriented, the duality, the extension, and the filtration rule. If $F_{C}$ is non\hyp zero, then $F_C$ must agree with Sarkar-Seed-Szabo formula.

         

\noindent

         

\begin{proof}
If $C$ is a connected resolution configuration, then by Lemma $3.3$ and Theorem $3.4$, $C$ must either be a tree or a dual tree. Lemma $3.3$ also ensures that $F_{C}$ will agree with Sarkar-Seed-Szabo formula.\\ \\
If $C$ is not connected, then \[C= \sqcup_{i=1}^{j} C_i\] where each $C_{i}$ is a $k_i$\hyp dimensional connected resolution configuration and $\sum_{i=1}^{j}k_i=k$. If each $C_i$ is either a tree or a dual tree, then the map
\begin{equation}
    G_C := \bigotimes_{i=1}^{j} F_{C_i}
\end{equation}
 as defined by the formulas (\ref{tree}) and (\ref{dual-tree}) satisfies all the rules and it agrees with Sarkar-Seed-Szabo formula. We claim that if $F_{C}$ is non\hyp zero, then $F_{C}=G_{C}$ as above and $F_{C}$ must be zero when $C$ is not a disjoint union of trees and dual trees. \\ \\ 
 Let's assume that $F_{C}(x) \neq 0$ for some monomial $x=x_1 \otimes \cdot \cdot \cdot \otimes x_j \in \bigotimes_{i=1}^{j} {V_{0}} (C_{i}) $ where $x_{i} \in V_{0}(C_{i})$. Expressing $F_{C}(x)$ as the sum of monomials in $V_1(C)$, we notice that $F_{C}(x)$ can be written as a linear sum of maps of the form $\bigotimes_{i=1}^{j} F_{C_{i}}(x_i)$. 
  We observe that if $F_C$ satisfies the filtration rule, then each $F_{C_i}$ appearing in the expansion will also satisfy the filtration rule. Let $x(P)$ and $y(P)$ denote starting circle and the ending circle containing the point $P$ in $C_i$ and $x(P)|x$. This implies, in particular $x(P)|x_i$. By the filtration rule, $y(P)|F_{C}(x)$. In particular, $y(P)$ divides each of the monomial in $F_{C}(x)$. Thus, we conclude $y(P)|F_{C_{i}}(x_i)$.
  These $F_{C_{i}}$ have the following property:
  \begin{equation*}
      F_{C_{i}}(X) 
      \begin{cases}
        \neq 0 ,\, \,  if X=x_{i} \\
        =0,  \, \, otherwise
      \end{cases}
  \end{equation*}

\noindent We can extend each $F_{C_i}$ to $\tilde{F}_{C_{i}}$ satisfying the naturality rule, the filtration rule and the duality rule by lemma 3.2. If the quantum grading of $\tilde{F}_{C_{m}}$ is strictly less than $2k_m$ for some $m$, then there will be another $\tilde{F}_{C_{n}}$ with quantum grading strictly greater than $2k_{n}$ for some $n$. Then by theorem $3.4$, it will imply that $\tilde{F}_{C_n} \equiv 0$ contrary to the hypothesis $F_{C_n} \neq 0 $. Hence, each of $F_{C_i}$ must have quantum grading $2k_i$. As $F_{C_{i}} \neq 0$, then $C_i$ must either be a tree or a dual tree. Then by lemma 3.3, we conclude that they must be equal to the formula $(\ref{tree})$ or $(\ref{dual-tree})$. 
\end{proof}
\end{theorem}
\noindent $\mathscr{C}$ denotes the collection of resolution configurations consisting of disjoint union of trees or dual trees. 
Next, we show that there does not exist a sub\hyp family $\mathscr{C}' \subset \mathscr{C}$ that suffices to define the total complex $CTot(L)$.  We notice that $\mathscr{C}$ is the disjoint union \[ \mathscr{C}= \sqcup_{i=1}^{\infty} \mathscr{C}_{i} \] where $\mathscr{C}_{i}$ is the collection of $i$\hyp dimensional resolution configurations consisting of disjoint union of trees or dual trees. We use induction on $i$ to show that the entire $\mathscr{C}_{i}$ is needed to define $CTot(L)$. The case $i=1$ is obvious as it is needed to define the Khovanov complex $CKh(L)$. Thus, we may assume that the entire $\mathscr{C}_{i}$ is needed for all $i \leq (n-1)$ to define $CTot(L)$. We will show that the $\mathscr{C}_{n}$ is also needed for $CTot(L)$. From now on, $C$ will always denote a resolution configuration which is a disjoint union of trees or dual trees and $F_{C}$ will denote the associated non\hyp zero map as in $(6)$. 
\begin{lemma}
There exists a $n$\hyp dimensional tree $C$ such that the associated map $F_{C} \neq 0$ in $CTot(L)$. 
\end{lemma}
\begin{proof}
Let $L$ denotes the $(n+1)$\hyp crossing knot whose (0,..,0) resolution looks like 
\begin{center}
\begin{tikzpicture}
\draw[blue] (0,0) circle [radius=0.5];
\draw[blue] (-1.5,-0.75) circle [radius=0.5];
\draw[blue] (-1.5,0.75) circle [radius=0.5];
\draw[blue] (1.5,-0.75) circle [radius=0.5];
\draw[blue] (-1.03,0.579)--(-0.4,0.3);
\draw[blue] (-1.05,-0.532)--(-0.4,-0.3);
\draw[blue] (1.05,-0.532)--(0.43,-0.255);
\draw[blue] (-1.08,-1.021)--(1.08,-1.021);
\node at (0,0) {$x_{n+1}$}; 
\node at (-1.5,-0.75) {$x_{1}$}; 
\node at (-1.5,0.75) {$x_{2}$}; 
\node at (1.5,-0.75) {$x_{n}$}; 
\node at (-0.75,-0.375+0.2) {$1$}; 
\node at (-0.75+0.1,0.375+0.3) {$2$}; 
\node at (0.75,-0.375+0.2) {$n$}; 
\node at (0,-1.25) {$n+1$}; 
\draw[blue] (1-0.5,1.75) circle [radius=0.05];
\draw[blue] (1.25-0.5,1.5) circle [radius=0.05];
\draw[blue] (1.5-0.5,1.25) circle [radius=0.05];
\draw[blue] (1.5,0.75) circle [radius=0.5];
\draw[blue] (0.44,0.237)--(1.07,0.495);
\node at (1.5,0.75) {$x_{n-1}$};
\node at (0,-2) {\textit{Figure 4. The resolution configuration associated to $L$}};
\end{tikzpicture}
\end{center}
Here the numbers $\{1,2,...,n+1 \}$ denote an ordering of the crossings of $L$.
The Maurer\hyp Cartan equation for $CTot(L)$ in degree $(n+1)$ takes the form \[ \sum_{i+j=n+1} [d_i:h_j]=0. \] It can be written as \[ \sum_{i+j=n+1, j <n}[d_i:h_j]=[d_1:h_n].\]
We notice that the left hand side does not depend on $h_{n}$. Thus, by induction hypothesis the left hand side can be computed using Sarkar-Seed-Szabo formula. We observe that the right hand side depends on the definition of $F_{C}$ where $C$ is a tree. We will show that $\left([d_1:h_n](\prod_{i=1}^{n+1}x_{i}) \neq 0 \right)$. This non\hyp zero term must be equal to the right hand side $\left(\sum_{i+j=n+1, j <n}[d_i:h_j](\prod_{i=1}^{n+1}x_{i}) \right)$. From this observation we shall conclude that $F_{C}$ must be non\hyp zero for some tree. \\ \\ 
We notice that there are two ending circles which we call $y_1$ and $y_2$. Thus, $[d_1:h_n]$ has the following form: 
\[[d_1:h_n]: \bigotimes_{i=1}^{n+1} V(x_{i}) \to V(y_1) \otimes V(y_2) \]
We next observe that, $\left(h_{n}d_{1}(\prod_{i=1}^{n+1}x_{i})=0\right)$. This is because the Khovanov differential $d_1$ sends the product of two circles in a join to zero and all the arcs appearing in $(0,0,...,0)$ resolution of the link $L$ are joins. \\ 
To compute $\left(d_1h_{n}(\prod_{i=1}^{n+1}x_{i})\right)$, we notice that there are exactly three $n$\hyp dimensional faces that contribute non\hyp zero in the computation  $\left(d_1h_{n}(\prod_{i=1}^{n+1}x_{i})\right)$. To be precise,
let $I=(0,0,...,0)$, $J_1=(0,1,1,...,1)$, $J_{2}=(1,1,...,1,0,1)$, $J_{3}=(1,1,...,1,0)$ and $K=(1,1,...,1)$ denote the specific coordinates in the Khovanov cube of $L$. Then the $n$\hyp dimensional faces $(I,J_1),(I,J_2)$ and $(I,J_3)$ are the only faces that are trees. Thus by Sarkar-Seed-Szabo formula, we have the following equation: 
\begin{align*}
    \left(d_1h_{n}\right)(\prod_{i=1}^{n+1}x_{i}) &= F_{J_1, K} F_{I, J_1}(\prod_{i=1}^{n+1}x_{i}) + F_{J_{2}, K} F_{I, J_2}(\prod_{i=1}^{n+1}x_{i})+ F_{J_{3}, K} F_{I, J_3}(\prod_{i=1}^{n+1}x_{i}) \\
    &= y_1y_2+y_1y_2+y_1y_2 \\
    &=y_1y_2
\end{align*}
Thus, we conclude that $F_{C}$ must be non\hyp zero for a $n$\hyp dimensional tree $C$. 
\end{proof}
\noindent Let $(C_1,x_1)$ and $(C_2,x_2)$ denote two resolution configurations which are disjoint union of trees and $x_i$ denotes a distinguished circle in $C_i$. Let $(C_{1},x_1)* (C_{2},x_2)$ denotes the resolution configuration which is constructed by taking disjoint union of $C_1$ and $C_2$ and identifying $x_1$ and $x_2$. We will abbreviate the notation as $C_1 *C_2$ when the distinguished circles are understood. We notice that $C_1 * C_2$ is also a resolution configuration consisting of disjoint union of trees.  Let $C$ denotes the resolution configuration $C=C_1 \sqcup C_2$. We prove the following lemma: 
\begin{lemma}
$F_C$ is non\hyp zero in Sarkar-Seed-Szabo formula if and only if $F_{C_1*C_2}$ is non\hyp zero. 
\end{lemma}
\begin{proof}
Let $k$ and $l$ denote the dimension of $C_1$ and $C_2$ respectively. Let $C'$ denotes the $k+l+1$ dimensional resolution configuration which looks like: 
\begin{center}
\begin{tikzpicture}
\draw[blue] (0,0) circle [radius=1]; 
\draw[blue,dotted] (2.5,0) circle [radius=2];
\draw[blue,dotted] (-2.5,0) circle [radius=2];
\draw[blue] (-2,1) circle [radius=0.5];
\draw[blue] (-2,-1) circle [radius=0.5];
\draw[blue] (-1.54,0.804)--(-0.8,0.6);
\draw[blue] (-1.54,-0.804)--(-0.8,-0.6);
\draw[blue]  (-3.25,1)circle [radius=0.25];
\draw[blue]  (-3.25,0)circle [radius=0.25];
\draw[blue] (-3.25,0.75)--(-3.25,0.25);
\draw[blue] (2,0.5) circle [radius=0.5]; 
\draw[blue] (2,-1) circle [radius=0.5];
\draw[blue] (2,0)--(2,-0.5);
\draw[blue] (0.9,0.436)--(1.5,0.5);
\node at (-2.5,-2.5) {$C_{1}$};
\node at (2.5,-2.5) {$C_{2}$};
\node at (0,-1.5) {$x_{1}=x_{2}$};
\draw[blue] (0,-1)--(0,1);
\node at (0.25,0) {$1$};
\node at (0,-3) {\textit{Figure 5: The resolution configuration} $C'$}; 
\end{tikzpicture}
\end{center}

\noindent Let $L$ denotes the $(k+l+1)$\hyp crossing link whose $(0,0,...,0)$ resolution looks like the resolution configuration as above. We enumerate the crossings of $L$ in such a fashion that the first crossing corresponds to the splitting arc as in the picture.
We will show that \[ [d_1:h_{k+l}]=0. \]
Let $\{\alpha_1,..,\alpha_{m} \}$ and $\{\beta_1,..,\beta_{n} \}$ denote the starting circles and the ending circles of $C$ respectively i.e. $\alpha_{i}$\hyp circles appear in the $(0,0,..,0)$ and $\beta_{i}$\hyp circles appear in the $(1,1,...,1)$ resolution of $L$. We have the associated map: \[\left([d_1:h_{k+l}]:\bigotimes_{i=1}^{m}V(x_i) \to \bigotimes_{j=1}^{n}V(y_{j}) \right) \] 
Let $I=(0,0,..,0)$, $J_{1}=(1,0,0,..,0)$, $J_{2}=(0,1,1,...,1)$ and $K=(1,1,...,1)$ denote the specific coordinates in the Khovanov cube of $L$. We notice that $d_1(\prod_{i=1}^{m}x_{i})$ is non\hyp zero only along the one dimensional face $(I,J_1)$. This is because $(I,J_1)$ is a split and the other faces are joins and Khovanov differential $d_1$ sends the product of two circles to zero for a join. Thus, we have the following equation: 
\[ \left(h_{k+l}d_{1}\right)(\prod_{i=1}^{m}x_{i})= F_{J_1,K}F_{I,J_1}(\prod_{i=1}^{m}x_{i})= \prod_{j=1}^{n}y_{j}. \]
We can similarly compute: 
\[\left(d_1h_{k+l}\right)(\prod_{i=1}^{m}x_{i})= F_{J_{2},K} F_{I,J_2}(\prod_{i=1}^{m}x_{i})= \prod_{j=1}^{n}y_{j}. \]
We observe that, $F_{J_1,K}= F_{C}$ and $F_{I,J_2}=F_{C_1*C_2}$. As, \[ \sum_{i+j=k+l+1, j<k+l} [d_i:h_j]=0. \] In this case, we conclude that $F_{C}$ is non\hyp zero if and only if $F_{C_1*C_2}$ is non\hyp zero. 
\end{proof}
\begin{lemma}
If $C$ is a $n$\hyp dimensional resolution configuration which is a disjoint union of trees or dually, if $C$ consists of disjoint union of dual trees, then $F_{C}$ is needed to define $CTot(.)$
\end{lemma}
\begin{proof}
By lemma $3.5$, we know that there exists a tree $C$ such that the associated map $F_{C}$ is needed in $CTot(.)$. By repeated application of lemma $3.6$, we conclude that $F_{C}$ is needed in $CTot(.)$ when $C$ is a disjoint union of trees. \\ \\ 
Dually, if $C$ consists of a disjoint union of dual trees, then $F_{C}$ is needed to define $CTot(.)$. 
\end{proof}

\noindent Finally, we will show that if $C$ is a disjoint union of trees and dual trees, then $F_{C}$ is also needed in $CTot(L)$. Before we go into the proof, we need to define \textbf{degree one circle} and \textbf{degree one arc}. \\
A circle $x$ in a resolution configuration $C$ is said to be a \textbf{degree one circle} if $x$ is connected to other circles of $C$ by exactly one arc. Dually, an arc $\gamma$ in a resolution configuration $C$ is said to be a \textbf{degree one arc} if its dual arc is connected to a degree one circle in its dual resolution configuration $C^{*}$.

\noindent Suppose $C=D_1 \sqcup D_2$ is a $n=k+l$\hyp  dimensional resolution configuration where $D_1 \, (\textit{resp.}\,D_2)$ is a $k \, (\textit{resp.} \,l)$\hyp dimensional resolution configuration which is a disjoint union of trees (\textit{resp.} dual trees).  We choose two circles $x_1 \in D_1$ and $z_1 \in D_2$ such that $x_1$ and $z_1$ can be joined by an arc $\gamma$, where $\gamma$ intersects $z_1$ in the region  defined by a \textbf{degree one arc} in $z_1$  that does not intersect any other arcs of $D_2$ (see. Figure 6). We call the \textbf{degree one arc} in $z_1$ by $\gamma^{'}$. Let $C'$ denotes the $n=k+l$\hyp dimensional resolution configuration which is obtained from $C \cup \{\gamma \}$ by doing a surgery along $\gamma^{'}$.  We notice that $C'$ is also a disjoint union of trees and dual trees ${D_1}' \sqcup {D_2}'$, where ${D_1}' (\textit{resp.${D_2}'$})$ is a disjoint union of trees (\textit{resp.} dual trees).
\begin{lemma}
$F_{C}$ is non\hyp zero if and only if $F_{C'}$ is non\hyp zero in Sarkar\hyp Seed\hyp Szabo construction. 
\begin{center}
\begin{tikzpicture}
\draw[blue] (0,0) circle [radius=1];
\draw[blue] (-2,-2) circle [radius=0.5];
\draw[blue] (0,-2) circle [radius=0.5];
\draw[blue] (-0.87,-0.493)--(-1.65,-1.643);
\draw[blue] (0,-1)--(0,-1.5);
\draw[blue] (4,0) circle [radius=1];
\draw[blue] (1,0)--(3,0);
\draw[blue] (3.15,0.527) to [out=-45, in=90] (3.5,0);
\draw[blue] (3.15,-0.527) to [out=45, in=-90] (3.5,0);
\draw[blue] (4,1)--(4,-1);
\draw[blue] (4.52,0.854) to [out=-120, in=180] (5,0);
\node at (0,0) {$x_1$};
\node at (2,0.25) {$\gamma$};
\node at (3.75,0) {$\gamma^{'}$};
\node at (4,-1.25) {$z_1$}; 
\node at (2,-0.25) {$1$}; 
\node at (3.25,0) {$2$};

\end{tikzpicture}
\end{center}
\begin{center}
\begin{tikzpicture}
\draw[blue] (6+2,0) circle [radius=1];
\draw[blue] (-2+6+2,-2) circle [radius=0.5];
\draw[blue] (6+2,-2) circle [radius=0.5];
\draw[blue] (-0.87+6+2,-0.493)--(-1.65+6+2,-1.643);
\draw[blue] (8,-1)--(8,-1.5);
\draw[blue] (10,0) circle [radius=0.5];
\draw[blue] (12,0) circle [radius=1];
\draw[blue] (12,1)--(12,-1);
\draw[blue] (12.52,0.854) to [out=-120, in=180] (13,0);
\draw[blue] (9,0)--(9.5,0);
\node at (8,0) {$x_1$}; 
\node at (9.25,0.25) {$\gamma$};
\node at (10,0) {$w_1$}; 
\node at (12,-1.25) {$w_2$}; 
\node at (10,-3) {Figure 6: $C \cup \{ \gamma \}$ \textit{and} $C'$}; 

\end{tikzpicture}
\end{center}
\begin{proof}

As before, we start with a $\left( k+l+1 \right)$\hyp crossing link diagram $L$ whose $(0,0,...,0)$ resolution looks like $C \cup \{\gamma \}$.  We enumerate the crossings in such a way that arc $\gamma$ corresponds to the first crossing and the arc $\gamma'$ corresponds to the second crossing of $L$. Let $w_1$ and $w_2$ denote the circles in $C'$ which are obtained from $C \cup \{\gamma \}$ by doing a surgery along $\gamma'$. Let $\{x_1,...,x_s \}$ denote the starting circles of $D_1$ and $\{y_1,...,y_t \}$ denote the ending circles of ${D_1}'$.  We will show that 
\[ \left([d_1:h_{k+l}] \right) (\prod_{i=1}^{s}x_i)=0. \]
Let $I=(0,0,..,0)$, $J_1=(0,1,0,...,0)$, $J_2=(0,1,1,..,1)$ and $K=(1,1,...,1)$ denote the specific coordinates in the Khovanov cube of $L$. In order to compute $\left(h_{k+l}d_1\right)(\prod_{i=1}^{s}x_i)$, we observe that it is non\hyp zero only for $F_{K,J_1} F_{J_1,I}$. This is because, doing surgery along any arcs other than $\gamma^{'}$ will not produce a resolution configuration that is a disjoint union of trees or dual trees. Thus, 
\begin{align*}
    \left(h_{k+l}d_1 \right) (\prod_{i=1}^{s}x_i) &= F_{J_1,K} F_{I,J_1} (\prod_{i=1}^{s}x_i) \\
&= F_{J_1,K}\left((\prod_{i=1}^{s}x_i) (w_1+w_2) \right) \\
&= F_{J_1,K} (x_1...x_s.w_1)+ F_{K,J_1}(x_1...x_s.w_2)\\
&= F_{C'}(x_1...x_s.w_1)+F_{C'}(x_1...x_s.w_2)\\
&= \prod_{i=1}^{t} y_{i} +0 \\ 
&= \prod_{i=1}^{t} y_{i}
\end{align*}
Similarly, we notice that to compute $\left(d_1h_{k+l}\right)(\prod_{i=1}^{s}x_i)$ we need to only consider $F_{J_2,K}F_{I,J_2}$. This is because none of $(k+l)$\hyp dimensional faces $(I,J)$ other than $(I,J_2)$ are a resolution configuration consisting of disjoint union of trees or dual trees. Thus, 
\begin{align*}
    \left(d_1h_{k+l} \right) (\prod_{i=1}^{s}x_i) &= F_{K,J_2}F_{J_2,I} (\prod_{i=1}^{s}x_i) \\ 
    &= F_{K,J_2}F_{C} (\prod_{i=1}^{s}x_i) \\
    &= \prod_{i=1}^{t} y_{i}
\end{align*}

\noindent By, the Maurer-Cartan equation for $CTot(L)$ in degree $(k+l+1)$ we notice that, 
 \[\sum_{i+j=k+l+1, j<k+l} [d_i: h_j](\prod_{i=1}^{s}x_i)= [d_1: h_{k+l}](\prod_{i=1}^{s}x_i)=0 \]
 This shows that $F_{C}$ is non\hyp zero if and only if $F_{C'}$ is non\hyp zero. 
\end{proof}
\end{lemma}
\noindent Finally as shown in the previous lemma, given a resolution configuration $C$, we can produce $C'$, $C^{''}$,..., $C^{'(l)}$ inductively such that $C^{'(l)}$ is a $(k+l)$\hyp dimensional resolution configuration which is a disjoint union of trees. By Lemma $3.7$, we know that $F_{C^{'(l)}} \neq 0$. Hence, $F_{C}$ must be non\hyp zero in Sarkar-Seed-Szabo construction. 

\noindent Finally, we have the following corollary: 
  \begin{corollary}
  Sarkar-Seed-Szabo's extended terms $ h_i, i\geq 2 $ are unique upto naturality, disoriented, extension, and the filtration rule. 
 \end{corollary}

 \begin{remark}
 We do not know if the rules are required for the uniqueness statement. We constructed a  map $\alpha$ of bidegree $(2,4)$ satisfying \[ [d_1: \alpha]=0 \] and not satisfying the $filtration$ $rule$. But, unfortunately $\alpha$ does not satisfy $[h_1: \alpha]=0$. Otherwise, $h_{2}' := h_2 + \alpha $ would be another possible construction of $h_2$ not satisfying the $filtration$ $rule$. 
 \end{remark}
 \begin{remark}
 The Maurer-Cartan equation for $CTot(L)$ at the homological grading $2$ takes the following form: 
 \begin{equation*}
     [h_1:h_2]=0 \, \, \, \textit{and} \, \, \, [d_1:h_2]+[d_2:h_1]=0 
\end{equation*}
 We can interpret the above equations as, $ad(h_1)(h_2)=0$ and $ad(d_1)(h_2)=[d_2:h_1]$. Now, as $h_1^2=0$ and $d_1^2=0$, we can look at the following complexes, $End(CTot(L), ad(h_1))$ and $End(CTot(L), ad(d_1))$ where $End(CTot(.))$ is the space of all $F_2$ linear endomorphisms of $CTot(L)$. Thus, the first condition allows us to conclude that $h_2$ is a homology class of the first complex. On the other hand we notice: 
 \begin{equation*}
  ad(d_1)([d_2:h_1])= [d_1:[d_2:h_1]]=0   
 \end{equation*}
 Thus, the existence of $h_2$ implies that $[d_2:h_1]$ represents the zero homology class in $End(CTot(L),ad(d_1))$. In conclusion we showed that the Szabo's geometric conditions forces us a choose a unique element in the above homology classes of the complex $CTot(L)$. Fixing the choice of $h_2$, we can write down the Maurer-Cartan equation for the homological grading $3$ as: 
 \begin{equation*}
    [h_1:h_3]+h_2^{2}=0 \, \, \, \textit{and} \, \, \, [d_1:h_3]+[d_2:h_2]+[d_3:h_1]=0 
 \end{equation*}
 As before, we can interpret the above equations as, $ad(h_1)(h_3)=h_2^{2}$ and $ad(d_1)(h_3)= [d_2:h_2]+[d_3:h_1]$. 
 We notice that, $ad(h_1)(h_2^{2})=0$ and $ad(d_1)([d_2:h_2]+[d_3:h_1])=0$. This is because, 
 \begin{align*}
 ad(d_1)([d_2:h_2]+[d_3:h_1]) &= [d_1:[d_2:h_2]]+[d_1:[d_3:h_1]]\\
 & = d_1(d_2h_2+h_2d_2) + (d_2h_2+h_2d_2)d_1 + d_1(d_3h_1+h_1d_3)\\ & +(d_3h_1+h_1d_3)d_1 \\ 
 &=\left(d_1d_2h_2 + d_2h_2d_1 \right) + \left(d_1h_2d_2+h_2d_2d_1\right) \\ 
 & + (d_1d_3h_1+d_3h_1d_1) + (d_1h_1d_3+h_1d_3d_1)\\ 
 &= (d_2d_1h_2+d_2h_2d_1) + (d_1h_2d_2+h_2d_1d_2) \\
 & + (d_1d_3h_1+d_3d_1h_1) + (h_1d_1d_3+h_1d_3d_1) \\ 
& \textit{as} \, \, [d_1:d_2]=0 \, \, \, and \, \, \, [h_1:d_1]=0 \\ 
&=d_2([d_1:h_2])+ ([d_1:h_2])d_2+ ([d_1:d_3])h_1 + h_1([d_1:d_3])\\
&= d_2([d_2:h_1])+([d_2:h_1])d_2+ d_2^{2}h_1+h_1d_2^{2} \\ 
&\textit{as} \, \, [d_1:d_3]+d_2^2=0 \, \, \,  \textit{and}\, \, \, [d_1:h_2]=[d_2:h_1] \\ 
&= 0 
 \end{align*}
Thus, the existence of $h_3$ implies $h_2^{2}$ as an element of $(End(CTot(L), ad(h_1))$ and $[d_2:h_2]+[d_3:h_1]$ as an element of $(End(CTot(L),ad(d_1)))$ represent the zero element in the respective homology groups. Szabo's geometric conditions allow us to choose a unique solution $h_3$. Inductively, $h_{n}$ can be thought of as a choice of the unique element in the complex $\left (CTot(L)\right)$. 
 \end{remark}

\bibliographystyle{plain} \bibliography{reference.bib} 

\Addresses

\end{document}